\documentclass[11pt]{article}

\usepackage[margin=1in]{geometry}

\usepackage{amsmath}
\usepackage{amssymb}
\usepackage{amsthm}
\usepackage{amsfonts}
\usepackage{enumerate}
\usepackage{graphicx}
\usepackage{url}
\usepackage[table]{xcolor}
\usepackage{soul}

\usepackage{hyperref}
\newcommand{\seqnum}[1]{\href{http://oeis.org/#1}{\underline{#1}}}

\usepackage{tikz}
\usetikzlibrary{arrows,shapes,positioning,decorations.markings,decorations.pathreplacing}
\tikzset{->-/.style={decoration={
  markings,
  mark=at position .5 with {\arrow{>}}},postaction={decorate}}}

\usepackage{array}
\newcolumntype{?}{!{\vrule width 3pt}}

\newtheorem{theorem}{Theorem}[section]

\newtheorem{lemma}[theorem]{Lemma}

\newtheorem{corollary}[theorem]{Corollary}
\newtheorem{proposition}[theorem]{Proposition}

\title{Raised $k$-Dyck Paths}

\author{
Paul Drube\\
\footnotesize Valparaiso University\\[-0.8ex] 
\footnotesize Valparaiso, Indiana, U.S.A. \\[-0.8ex]
\footnotesize\tt paul.drube@valpo.edu
}

\begin{document}

\maketitle

\begin{abstract}
Raised $k$-Dyck paths are a generalization of $k$-Dyck paths that may both begin and end at a nonzero height.  In this paper, we develop closed formulas for the number of raised $k$-Dyck paths from $(0,\alpha)$ to $(\ell,\beta)$ for all height pairs $\alpha,\beta \geq 0$, all lengths $\ell \geq 0$, and all $k \geq 2$.  We then enumerate raised $k$-Dyck paths with a fixed number of returns to ground, a fixed minimum height, and a fixed maximum height, presenting generating functions (in terms of the generating functions $C_k(t)$ for the $k$-Catalan numbers) when closed formulas aren't tractable.  Specializing our results to $k=2$ or to $\alpha < k$ reveal connections with preexisting results concerning height-bounded Dyck paths and ``Dyck paths with a negative boundary", respectively.
\end{abstract}

\section{Introduction}
\label{sec: intro}

For any $k \geq 2$, a $k$-Dyck path of length $\ell$ and height $h$ is an integer lattice path from $(0,0)$ to $(\ell,h)$ that uses steps $\lbrace U = (1,1), D = (1,1-k)) \rbrace$ and stays weakly above the line $y=0$.  One may verify that the terminal point of any $k$-Dyck path must satisfy $\ell = h \bmod k$.  Thus we restrict our attention to $k$-Dyck paths of length $kn+h$ and height $h$, denoting the collection of all $k$-Dyck paths of length $kn+h$ and height $h$ by $\mathcal{D}^k_{n,h}$.\footnote{$k$-Dyck paths of length $kn$ and height $km$ are often referred to as $k$-Dyck paths of ``semi-length" $n$ and ``semi-height" $m$, with $\mathcal{D}^k_{n,m}$ also sometimes being used to refer to such paths.}

It is well known that $k$-Dyck paths of length $kn$ and height $0$ are enumerated by the $k$-Catalan numbers (or Fuss-Catalan numbers), a one-parameter generalization of the Catalan numbers given by $C_n^k = \frac{1}{kn+1} \binom{kn+1}{n}$ for all $k \geq 2$ and $n \geq 0$.  In particular, $\vert \mathcal{D}^k_{n,0} \vert = C_n^k$ for all $k \geq 2$ and $n \geq 0$.  In the case of $k=2$, this corresponds to the classic combinatorial interpretation of the Catalan numbers by Dyck paths of length $2n$ and height $0$.  For more information about the $k$-Catalan numbers and their combinatorial interpretations, see Hilton and Pedersen \cite{HP} or Heubach, Li and Mansour \cite{HLM}.  For even more details about the classic Catalan numbers, see Stanley \cite{Stanley}.

Now let $C_k(t) = \sum_{n=0}^\infty C^k_n t^n$ be the ordinary generating function for the $k$-Catalan numbers.  As shown by Hilton and Pedersen \cite{HP}, the $k$-Catalan numbers satisfy $C^k_{n+1} = \sum_{i_1 + \hdots + i_k = n} C_{i_1} \cdots C_{i_k}$ for all $n \geq 0$, implying that these generating functions obey $C_k(t) = t C_k(t)^k + 1$.  If we use $[t^n]p(t)$ to denote the coefficient of $t^n$ in the power series $p(t)$, another standard result asserts $\vert \mathcal{D}^k_{n,h} \vert = [t^n]C_k(t)^{h+1}$ for all $n,h \geq 0$.  See Figure \ref{fig1: intro height decomposition} for the decomposition that yields this result.

\begin{figure}[ht!]
\centering
\begin{tikzpicture}
	[scale=.45,auto=left,every node/.style={circle,fill=black,inner sep=0pt,outer sep=0pt}]
	\draw[fill=gray!20!](0,0) rectangle (4,2);
	\node[fill=none] (b1) at (2,1) {\scalebox{1.2}{$P_0$}};
	\draw[thick] (4,0) to (5,1);
	\draw[fill=gray!20!](5,1) rectangle (9,3);
	\node[fill=none] (b1) at (7,2) {\scalebox{1.2}{$P_1$}};	
	\draw[thick] (9,1) to (10,2);
	\node[fill=none] (dots) at (11.5,2) {$\cdots$};
	\draw[fill=gray!20!](13,2) rectangle (17,4);
	\node[fill=none] (b1) at (15,3) {\scalebox{1.2}{$P_h$}};
	\node[inner sep=1pt] (1) at (0,0) {};
	\node[inner sep=1pt] (2) at (4,0) {};
	\node[inner sep=1pt] (3) at (5,1) {};
	\node[inner sep=1pt] (4) at (9,1) {};
	\node[inner sep=1pt] (5) at (10,2) {};
	\node[inner sep=1pt] (6) at (13,2) {};	
	\node[inner sep=1pt] (7) at (17,2) {};
\end{tikzpicture}
\caption{A $k$-Dyck path $P$ of height $h$ decomposed into a sequence of $h+1$ paths $P_i$ of height $0$, according to the rightmost $U$ steps at each height.  Note that some of the $P_i$ may be empty.}
\label{fig1: intro height decomposition}
\end{figure}
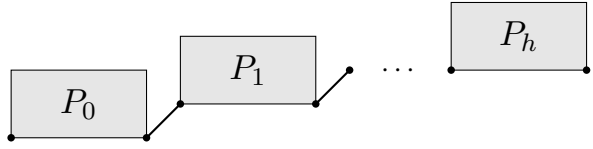

Also proven by Hilton and Pedersen \cite{HP} is that $[t^n]C_k(t)^r = \frac{r}{kn+r} \binom{kn+r}{n} = R_{k,r}(n)$ for all $k \geq 2$, $n \geq 0$, and $r \geq 1$.  Here the notation $R_{k,r}(n)$ corresponds to the Raney number (two-parameter Fuss-Catalan number).  This gives

\begin{equation}
\label{eq1: shape (0,h) closed formula}
\vert \mathcal{D}_{n,h}^k \vert = [t^n] C_k(t)^{h+1} = \frac{h+1}{kn+h+1} \binom{kn+h+1}{n}.
\end{equation}

The primary goal of this paper is to generalize the closed formula of \eqref{eq1: shape (0,h) closed formula} to generalized $k$-Dyck paths that may begin (as well as end) at any non-zero height, objects that we informally refer to as ``raised $k$-Dyck paths".  These raised $k$-Dyck paths may be interpreted as a natural generalization of the ``$k$-Dyck paths with negative boundary" (or $k_t$-Dyck paths) recently investigated by Selkirk \cite{Selkirk}, Asinowski, Hackl and Selkirk \cite{AHS}, and Prodinger \cite{Prodinger}, although our results are developed in such a manner that we needn't restrict our attention to starting heights less than $k$.  We then develop closed formulas for the number of raised $k$-Dyck paths with a fixed minimum height and a fixed number of returns.  To close the paper, we use our results to derive new generating functions for the number of $k$-Dyck paths of bounded height, a topic where all previous investigations appear to be limited to the $k=2$ case or are broadly theoretical and don't account for general starting/ending heights (see \cite{BP}, \cite{Bacher} for recent discussions concerning $k$-Dyck paths of bounded height).

\section{Raised k-Dyck paths}
\label{sec: raised Dyck paths}

Once again fix $k \geq 2$.  For any $\alpha,\beta \geq 0$, a raised $k$-Dyck path of length $\ell$ and shape $(\alpha,\beta)$ is an integer lattice path from $(0,\alpha)$ to $(\ell,\beta)$ that uses steps $\lbrace U=(1,1), D=(1,1-k) \rbrace$ and stays weakly above the line $y=0$.  The terminal point of any such path must satisfy $\ell = (\beta - \alpha) \bmod k$, justifying our restriction to $k$-Dyck paths of length $kn+\beta-\alpha$ and shape $(\alpha,\beta)$.  Denote the set of all $k$-Dyck paths of length $kn+\beta-\alpha$ and shape $(\alpha,\beta)$ by $\mathcal{D}^k_{n,(\alpha,\beta)}$, and then define $\vert \mathcal{D}^k_{n,(\alpha,\beta)} \vert = C^k_{n,(\alpha,\beta)}$.  Notice that all elements of $\mathcal{D}^k_{n,(\alpha,\beta)}$ contain precisely $n+\beta-\alpha$ up steps and $n$ down steps, meaning that the ``n index" of a particular path corresponds to its number of $D$ steps.

It is clear that $\mathcal{D}^k_{n,(0,\beta)} = \mathcal{D}^k_{n,\beta}$.  It is also clear that $\mathcal{D}^k_{n,(\beta,\beta)}$ is in bijection with integer lattice paths from $(0,0)$ to $(kn,0)$ that use step set $\lbrace U,D \rbrace$ and stay weakly above the line $y = -\beta$.  Horizontal reflection then places $\mathcal{D}^k_{n,(\beta,\beta)}$ in bijection with the $k_\beta$-Dyck paths of Selkirk \cite{Selkirk} and Asinowski, Hackl and Selkirk \cite{AHS}.  More generally, whenever $\alpha \geq \beta$, the set $\mathcal{D}^k_{n,(\alpha,\beta)}$ is in bijection with generalized $k$-Dyck paths from $(0,0)$ to $(kn + \beta,0)$ that stay weakly above $y=-\beta$ and begin with at least $\alpha$ consecutive $U$ steps.  This gives additional bijections between our sets and some of the $k_\beta$-Dyck paths studied by Prodinger \cite{Prodinger}.

Before proceeding, notice that a trivial path of length $\ell=0$ only exists when $n=0$ and $\alpha=\beta$.  In this case we have $C^k_{0,(\beta,\beta)}=1$.  Also note that $C^k_{n,(\alpha,\beta)}=0$ whenever $kn + \beta - \alpha < 0$.  This corresponds to the fact that any element of $\mathcal{D}^k_{n,(\alpha,\beta)}$ with $\alpha > \beta$ requires some minimal number of $D$ steps in order to end at the correct height.

Fundamental to much of our approach is the following recurrence.  In this and subsequent results, we automatically set $C^k_{n,(\alpha,\beta)} = 0$ whenever $\alpha < 0$ or $\beta < 0$.

\begin{proposition}
\label{thm2: basic recurrence}
For any $k \geq 2$, $n \geq 0$, and $\alpha,\beta \geq 0$ other than $n=0$ and $\alpha=\beta$,

$$C^k_{n,(\alpha,\beta)} = C^k_{n,(\alpha+1,\beta)} + C^k_{n-1,(\alpha-k+1,\beta)}.$$
\end{proposition}
\begin{proof}

Observe that $n=0$ and $\alpha=\beta$ corresponds to trivial paths, which cannot be decomposed as outlined below.  Excepting that case, let $S_U$ be the subset of $\mathcal{D}^k_{n,(\alpha,\beta)}$ including all paths that begin with a $U$ step, and let $S_D$ be the subset of $\mathcal{D}^k_{n,(\alpha,\beta)}$ including all paths that begin with a $D$ step.  Eliminating the first step of every $P \in S_U$ gives a path of length $kn+\beta-\alpha-1 = kn+\beta-(\alpha+1)$ and shape $(\alpha+1,\beta)$, placing $S_U$ in bijection with $\mathcal{D}^k_{n,(\alpha+1,\beta)}$.  Eliminating the first step of every $P \in S_D$ gives a path of length $kn+\beta-\alpha-1 = k(n-1)+\beta-(\alpha-k+1)$ and shape $(\alpha-k+1,\beta)$, placing $S_D$ in bijection with $\mathcal{D}^k_{n-1,(\alpha-k+1,\beta)}$.
\end{proof}

Fully utilizing the recurrence of Proposition \ref{thm2: basic recurrence} requires multivariate generating functions.  Simultaneously accounting for all shapes $(\alpha,\beta)$, define $C_k(q,r,t) = \sum_{\alpha,\beta,n \geq 0} C^k_{n,(\alpha,\beta)} q^\alpha r^\beta t^n$.  For reasons that will become clear in upcoming sections, we separately denote the ordinary generating function for paths of fixed shape $(\alpha,\beta)$ by $C_{k,(\alpha,\beta)}(t) = \sum_{n \geq 0} C^k_{n,(\alpha,\beta)} t^n = [q^\alpha r^\beta] C_k(q,r,t)$.

For fixed shape $(\alpha,\beta)$, observe that the order of $C_{k,(\alpha,\beta)}(t)$ is the smallest non-negative integer $n$ that such $n \geq \frac{\alpha-\beta}{k-1}$, corresponding to the minimal number of $D$ steps in any path of shape $(\alpha,\beta)$.  In particular, if $\alpha \leq \beta$ then $C_{k,(\alpha,\beta)}(t)$ has order $0$.  When $\alpha \leq \beta$, the minimal coefficient is always $[t^0]C_{k,(\alpha,\beta)}(t) = 1$, corresponding to the unique path of shape $(\alpha,\beta)$ with zero $D$ steps.  When $\alpha > \beta$, the minimal coefficient of $C_{k,(\alpha,\beta)}(t)$ may or may not be $1$.

Proposition \ref{thm2: basic recurrence} may be used to derive the following relationship for $C_k(q,r,t)$:

\begin{theorem}
\label{thm2: generating functions}
For any $k \geq 2$,

$$C_k(q,r,t) = \frac{\sum_{i \geq 0} \left( C_k(t)^{i+1} - q^{i+1} \right) r^i}{1-q+q^k t}.$$
\end{theorem}
\begin{proof}
The recurrence of Proposition \ref{thm2: basic recurrence} is equivalent to $C^k_{n,(\alpha,\beta)} = C^k_{n,(\alpha-1,\beta)} - C^k_{n-1,(\alpha-k,\beta)}$ for all $n \geq 0$ and $\alpha \geq 1$.  This suggests a relation that includes $C_k(q,r,t) = q \kern+1pt C_k(q,r,t) - q^k t \kern+1pt C_k(q,r,t)$.  Accounting for the $\alpha=0$ case, where shape $(0,\beta)$ paths are generated by $C_k(t)^{\beta+1}$, requires an additional $\sum_{i \geq 0} C_k(t)^{i+1} r^i$ term on the right side.  Also accounting for the trivial case of $n=0$ and $\alpha=\beta$, to which Proposition \ref{thm2: basic recurrence} doesn't apply, we have the full recurrence

\begin{equation}
\label{eq2: generating function recurrence}
C_k(q,r,t) = q \kern+1pt C_k(q,r,t) - q^k t \kern+1pt C_k(q,r,t) + \sum_{i \geq 0} C_k(t)^{i+1} r^i - \sum_{i \geq 0} q^{i+1} r^i.
\end{equation}
\end{proof}

The generating function of Theorem \ref{thm2: generating functions} may be used to derive closed formulas for all of the $C^k_{n,(\alpha,\beta)}$, regardless of starting height.  In all that follows, we set $\binom{a}{b} = 0$ whenever $a < 0$ or $b < 0$.

\begin{theorem}
\label{thm2: closed formula}
For any $k \geq 2$ and $n,\alpha,\beta \geq 0$,

\begin{center}
\scalebox{.89}{
$\displaystyle{C^k_{n,(\alpha,\beta)} = \left( \sum_{i \geq 0} (-1)^i \frac{\beta+1}{k(n-i)+\beta+1} \binom{k(n-i)+\beta+1}{n-i} \binom{\alpha-(k-1)i}{i} \right) - (-1)^n \binom{\alpha-\beta-1-(k-1)n}{n}}$.}
\end{center}
\end{theorem}
\begin{proof}
Specializing the formula of Theorem \ref{thm2: generating functions} to fixed $\beta$ gives
\begin{equation}
\label{eq2: closed formula proof 1}
[r^\beta] C_k(q,r,t) = \frac{C_k(t)^{\beta+1}-q^{\beta+1}}{1-q+q^kt} = \left( C_k(t)^{\beta+1}-q^{\beta+1} \right) \left( 1 + (q-q^kt) + (q-q^kt)^2 + \hdots \right).
\end{equation}
One may verify that the coefficient of $q^\alpha$ in $(1 + (q-q^kt) + (q-q^kt)^2 + \hdots)$ is $\sum_{i \geq 0} (-1)^i \binom{\alpha - (k-1)i}{i} t^i$.  This implies
\begin{equation}
\label{eq2: closed formula proof 2} 
[q^\alpha r^\beta]C_k(q,r,t) = C_k(t)^{\beta+1} \sum_{i \geq 0} (-1)^i \binom{\alpha - (k-1)i}{i} t^i - \sum_{i \geq 0} (-1)^i \binom{\alpha - \beta - 1 - (k-1)i}{i} t^i.
\end{equation}
As noted in Section \ref{sec: intro}, $C_k(t)^{\beta+1}$ may be rewritten as $C_k(t)^{\beta+1} = \sum_{i \geq 0} \frac{\beta+1}{ki+\beta+1} \binom{ki+\beta+1}{i} t^i$.  This transforms the first term from the right side of \eqref{eq2: closed formula proof 2} into a convolution of two power series.  Extracting the coefficient of $q^\alpha$ from both terms of \eqref{eq2: closed formula proof 2} yields our formula for $[q^\alpha r^\beta t^n] C_k(q,r,t) = C^k_{n,(\alpha,\beta)}$.
\end{proof}

Inspecting the formula of Theorem \ref{thm2: closed formula}, observe that the trailing term can only be nonzero when $\alpha > \beta$.  Also, at least one of the binomial coefficients from each term of the summation is zero unless $i \leq \min(n,\frac{\alpha}{k})$.

All of this means that the formula of Theorem \ref{thm2: closed formula} is much simpler when the starting height $\alpha$ is relatively small.  In particular, when $0 \leq \alpha \leq k-1$, the leading summation contains only a single nonzero term and we have the following.

\begin{corollary}
\label{thm2: closed formula corollary 1}
For any $k \geq 2$, $\beta \geq 0$, and $0 \leq \alpha \leq k-1$,

$$C^k_{n,(\alpha,\beta)} =
\begin{cases}
\frac{\beta+1}{kn+\beta+1} \binom{kn+\beta+1}{n} = R_{k,\beta+1}(n) & \text{ if } n>0,\\
1 & \text{ if } n=0 \text{ and } \alpha \leq \beta,\\
0 & \text{ if } n=0 \text{ and } \alpha > \beta.
\end{cases}$$
\end{corollary}
\begin{proof}
When $n > 0$ and $\alpha \leq k-1$, the final term from Theorem \ref{thm2: closed formula} is always zero and the leading summation simplifies to a single term.  When $n = 0$, Theorem \ref{thm2: closed formula} gives $\frac{\beta+1}{\beta+1} \binom{\beta+1}{0} - \binom{\alpha-\beta-1}{0}$.
\end{proof}

Still restricting our attention to $0 \leq \alpha \leq k-1$, we can alternatively begin with \eqref{eq2: closed formula proof 2} to recast Corollary \ref{thm2: closed formula corollary 1} is terms of the generating functions $C_{k,(\alpha,\beta)}(t) = [q^\alpha r^\beta] C_k(q,r,t)$:

\begin{corollary}
\label{thm2: closed formula corollary 2}
For any $k \geq 2$, $\beta \geq 0$, and $0 \leq \alpha \leq k-1$,

$$C_{k,(\alpha,\beta)}(t) =
\begin{cases}
C_k(t)^{\beta+1} & \text{ if } \alpha \leq \beta,\\
C_k(t)^{\beta+1} - 1 & \text{ if } \alpha > \beta.
\end{cases}
$$
\end{corollary}
\begin{proof}
By \eqref{eq2: closed formula proof 2}, when $0 \leq \alpha \leq k-1$ we have $[q^\alpha r^\beta]C_k(q,r,t) = C_k(t)^{\beta+1} \binom{\alpha}{0} - \binom{\alpha-\beta-1}{0}$.
\end{proof}

Corollaries \ref{thm2: closed formula corollary 1} and \ref{thm2: closed formula corollary 2} place our work in agreement with Selkirk \cite{Selkirk} and Asinowski, Hackl and Selkirk \cite{AHS}, assuming we restrict ourselves to their range of $0 \leq \alpha \leq k-1$.  In this case, observe that Corollaries \ref{thm2: closed formula corollary 1} and \ref{thm2: closed formula corollary 2} may also be proven by placing $\mathcal{D}^k_{n,(\alpha,\beta)}$ in bijection with $\mathcal{D}^k_{n,(0,\beta)}$ via the map that adds $\alpha$ consecutive $U$ steps to the beginning of every $P \in \mathcal{D}^k_{n,(\alpha,\beta)}$.  This bijection fails when $\alpha > k-1$, since it is no longer the case that every $P \in D^k_{n,(0,\beta)}$ must begin with $\alpha > k-1$ consecutive $U$ steps.

Computation of $C_{k,(\alpha,\beta)}(t)$ becomes increasingly difficult as one extends above $\alpha = k-1$.  See Appendix \ref{sec: appendix} for a comparison of the sequences generated by $C_{k,(\alpha,0)}(t)$ to previously-catalogued sequences on OEIS \cite{OEIS}, for small $k \geq 2$ and various shapes $(\alpha,\beta)$.

\subsection{Raised k-Dyck paths, k=2 case}
\label{subsec: raised Dyck paths k=2}

As with most combinatorial objects related to the $k$-Catalan numbers, investigating raised $k$-Dyck paths becomes much easier in the case of $k=2$.  In this subsection, we present a series of results involving the $C^k_{n,(\alpha,\beta)}$ that hold only when $k=2$.

The primary reason the $k=2$ case is simpler is the fact that the left-right reflection of a raised $2$-Dyck path still qualifies as a raised $2$-Dyck path.  In particular, reflecting a $2$-Dyck path of length $2n+\beta-\alpha$ and shape $(\alpha,\beta)$ results in a $2$-Dyck path of length $2n+\beta-\alpha = 2(n+\beta-\alpha)+\alpha-\beta$ shape $(\beta,\alpha)$.  In terms of generating functions, this prompts:

\begin{proposition}
\label{thm2: k=2 horizontal reflection}
For all $\alpha,\beta \geq 0$,
$$C_{2,(\beta,\alpha)}(t) = t^{\beta-\alpha} C_{2,(\alpha,\beta)}(t).$$
\end{proposition}

Notice that Proposition \ref{thm2: k=2 horizontal reflection} holds even if $\beta - \alpha < 0$.  If $\alpha > \beta$, then $C_{2,(\alpha,\beta)}(t)$ has order $\alpha - \beta$ and $t^{\beta-\alpha} C_{2,(\alpha,\beta)}(t)$ is a valid (order $0$) power series.  When dealing with the $k=2$ case, Proposition \ref{thm2: k=2 horizontal reflection} allows us to restrict our attention to shapes $(\alpha,\beta)$ where $\beta \geq \alpha$.

Our next result is a replacement of the generating function equation \eqref{eq2: closed formula proof 2} from the proof of Theorem \ref{thm2: closed formula} that holds only when $k=2$.

\begin{theorem}
\label{thm2: k=2 generating function}
For all $\alpha,\beta \geq 0$,
$$C_{2,(\alpha,\beta)}(t) = \sum_{i=0}^{\min(\alpha,\beta)} t^{\alpha-i} C_2(t)^{\alpha+\beta+1-2i}.$$
\end{theorem}
\begin{proof}
For each $n \geq 0$, we partition $\mathcal{D}^2_{n,(\alpha,\beta)}$ into sets $\mathcal{S}_{n,0},\hdots,\mathcal{S}_{n,\min(\alpha,\beta)}$, where $\mathcal{S}_{n,i}$ includes all paths whose lowest point lies along $y=i$.  As shown in Figure \ref{fig2: k=2 decomposition}, every path $P \in S_{i,n}$ may be decomposed into a sequence of $(\alpha-i)+(\beta-i)+1$ subpaths of shape $(0,0)$.  Notice that this decomposition includes $\alpha-i$ ``external" down steps that aren't included within one of the shape-$(0,0)$ subpaths.  If we define the generating function $S_i(t) = \sum_{n \geq 0} \vert \mathcal{S}_{n,i} \vert t^n$, this decomposition implies that $S_i(t)= t^{\alpha-i} C_2(t)^{\alpha+\beta+1-2i}$.
\end{proof}

\begin{figure}[ht!]
\centering
\begin{tikzpicture}
	[scale=.4,auto=left,every node/.style={circle,fill=black,inner sep=0pt,outer sep=0pt}]
	\draw (-2,-1) to (31,-1);
	\node[fill=none] (b0) at (-0.5,3) {\scalebox{1.2}{$\cdots$}};
	\draw[fill=gray!20!](1,3) rectangle (4,6);
	\node[fill=none] (b1) at (2.5,4.5) {\scalebox{1}{$(0,0)$}};
	\draw[thick] (4,3) to (5,2);
	\draw[fill=gray!20!](5,2) rectangle (8,5);
	\node[fill=none] (b2) at (6.5,3.5) {\scalebox{1}{$(0,0)$}};
	\draw[thick] (8,2) to (9,1);
	\node[fill=none] (b3) at (10.5,1) {\scalebox{1.2}{$\cdots$}};	
	\draw[thick] (12,1) to (13,0);
	\draw[fill=gray!20!](13,0) rectangle (16,3);
	\node[fill=none] (b4) at (14.5,1.5) {\scalebox{1}{$(0,0)$}};
	\draw[thick] (16,0) to (17,1);	
	\node[fill=none] (b5) at (18.5,1) {\scalebox{1.2}{$\cdots$}};
	\draw[thick] (20,1) to (21,2);
	\draw[fill=gray!20!](21,2) rectangle (24,5);
	\node[fill=none] (b6) at (22.5,3.5) {\scalebox{1}{$(0,0)$}};
	\draw[thick] (24,2) to (25,3);	
	\draw[fill=gray!20!](25,3) rectangle (28,6);
	\node[fill=none] (b7) at (26.5,4.5) {\scalebox{1}{$(0,0)$}};
	\node[fill=none] (b8) at (29.5,3) {\scalebox{1.2}{$\cdots$}};
	\node[inner sep=1pt] (1) at (1,3) {};
	\node[inner sep=1pt] (2) at (4,3) {};
	\node[inner sep=1pt] (3) at (5,2) {};
	\node[inner sep=1pt] (4) at (8,2) {};
	\node[inner sep=1pt] (5) at (9,1) {};
	\node[inner sep=1pt] (6) at (12,1) {};
	\node[inner sep=1pt] (7) at (13,0) {};
	\node[inner sep=1pt] (8) at (16,0) {};
	\node[inner sep=1pt] (9) at (17,1) {};
	\node[inner sep=1pt] (10) at (20,1) {};
	\node[inner sep=1pt] (11) at (21,2) {};
	\node[inner sep=1pt] (12) at (24,2) {};
	\node[inner sep=1pt] (13) at (25,3) {};
	\node[inner sep=1pt] (14) at (28,3) {};	
\end{tikzpicture}
\caption{The decomposition of a path $P \in \mathcal{D}^2_{n,(\alpha,\beta)}$ into a sequence of $(\alpha-i)+(\beta-i)+1$ subpaths of shape $(0,0)$, as referenced in the proof of Theorem \ref{thm2: k=2 generating function}.}
\label{fig2: k=2 decomposition}
\end{figure}

If $\beta \geq \alpha$, the formula of Theorem \ref{thm2: k=2 generating function} may be rewritten as $C_{2,(\alpha,\beta)}(t) = \sum_{i=0}^{\alpha} t^i C_2(t)^{\beta-\alpha+2i+1}$.  Similarly, if $\alpha \geq \beta$, Theorem \ref{thm2: k=2 generating function} may be rewritten as $C_{2,(\alpha,\beta)}(t) = \sum_{i=0}^{\beta} t^{\alpha-\beta+i} C_2(t)^{\alpha-\beta+2i+1}$.  Together these identities ensure $C_{2,(\beta,\alpha)}(t) = t^{\beta-\alpha} C_{2,(\alpha,\beta)}(t)$, placing Theorem \ref{thm2: k=2 generating function} in agreement with Proposition \ref{thm2: k=2 horizontal reflection}.  

Temporarily restricting our attention to the case of $\beta \geq \alpha$, also note that we may use the identity $C_2(t) = t \kern+1pt C_2(t)^2 + 1$ to rewrite the formula above as

\begin{equation}
\label{eq2: k=2 generating function relations}
C_{2,(\alpha,\beta)}(t) = C_2(t)^{\beta-\alpha+1} \sum_{i=0}^\alpha (t \kern+1pt C_2(t)^2)^i = C_2(t)^{\beta-\alpha+1} \sum_{i=0}^\alpha (C_2(t)-1)^i.
\end{equation}

More significantly, Theorem \ref{thm2: k=2 generating function} may used to develop a closed formula for arbitrary $C^2_{n,(\alpha,\beta)}$, giving a simpler replacement of Theorem \ref{thm2: closed formula} that holds only when $k=2$.

\begin{theorem}
\label{thm2: k=2 closed formula}
For all $n,\alpha,\beta \geq 0$,
$$C^2_{n,(\alpha,\beta)} = \sum_{i=0}^{\min(\alpha,\beta)} \frac{\alpha+\beta+1-2i}{2n+\beta-\alpha+1} \binom{2n+\beta-\alpha+1}{n-\alpha+i}.$$
\end{theorem}
\begin{proof}
Fixing $n \geq 0$ and applying the definition of Raney numbers, we have $[t^n] t^{\alpha-i} C_2(t)^{\alpha+\beta+1-2i} = [t^{n-\alpha+i}]C_2(t)^{\alpha+\beta+1-2i} = \frac{\alpha+\beta+1-2i}{2(n-\alpha+i)+(\alpha+\beta+1-2i)} \binom{2(n-\alpha+i)+(\alpha+\beta+1-2i}{n-\alpha+i}$.  Our closed formula for the $C^2_{n,(\alpha,\beta)} = [t^n]C_{2,(\alpha,\beta)}(t)$ then follows from the summation of Theorem \ref{thm2: k=2 generating function}.
\end{proof}

\section{Raised k-Dyck paths, filtered by minimum height and returns}
\label{sec: raised Dyck paths by returns}

For the rest of this paper, we focus upon the enumeration of raised $k$-Dyck paths that satisfy additional conditions.  We begin by developing formulas for the number of paths $P \in \mathcal{D}^k_{n,(\alpha,\beta)}$ that have a fixed minimum height and paths $P \in \mathcal{D}^k_{n,(\alpha,\beta)}$ that have a certain number of ``returns to ground".  Enumerating raised $k$-Dyck paths that have a fixed maximum height is delayed until Section \ref{sec: Dyck paths of bounded height}.

\subsection{Raised k-Dyck paths, by minimum height}
\label{subsec: raised Dyck paths by minimum height}

For traditional $k$-Dyck paths, all of which necessarily begin at height $y=0$, it is unnecessary to categorize paths according to their minimum $y$-coordinate.  For raised $k$-Dyck paths of shape $(\alpha,\beta)$, this question becomes non-trivial when both $\alpha > 0$ and $\beta > 0$.

Take any path $P \in \mathcal{D}^k_{n,(\alpha,\beta)}$.  If $P$ stays weakly above $y=m$, we say that $P$ is bounded from below by $m$.  Then let $\mathcal{L}^{k,m}_{n,(\alpha,\beta)}$ denote the collection of all $P \in \mathcal{D}^k_{n,(\alpha,\beta)}$ that are bounded from below by $m$.  For any such set, there exists a clear bijection between $\mathcal{L}^{k,m}_{n,(\alpha,\beta)}$ and $\mathcal{D}^k_{n,(\alpha-m,\beta-m)}$ whereby paths in $\mathcal{L}^{k,m}_{n,(\alpha,\beta)}$ are shifted $m$ units downward.  As such, we focus upon enumerating paths that actually obtain a fixed minimum height.

So once again take $P \in \mathcal{D}^k_{n,(\alpha,\beta)}$.  If $P$ is bounded from below by $m$ yet is not bounded from below by $m+1$, meaning that $m$ is the minimum $y$-coordinate among all points $(x_i,y_i)$ along $P$, we say that $P$ has a minimum height of $m$.  Then let ${_m}\mathcal{D}^k_{n,(\alpha,\beta)}$ to denote the set of all raised $k$-Dyck paths of length $kn+\beta-\alpha$ and shape $(\alpha,\beta)$ with minimum height $m$, and set $\vert {_m}\mathcal{D}^k_{n,(\alpha,\beta)} \vert = {_m}C^k_{n,(\alpha,\beta)}$.  For fixed shape $(\alpha,\beta)$ and fixed $m$, define the generating function ${_m}C_{k,(\alpha,\beta)}(t) = \sum_{n \geq 0} {_m}C^k_{n,(\alpha,\beta)} t^n$.

Obviously, all $P \in \mathcal{D}^k_{n,(\alpha,\beta)}$ have a minimum height that falls in the range $0 \leq m \leq \min(\alpha,\beta)$.  It follows that $\mathcal{D}^k_{n,(\alpha,\beta)} = \bigcup_{i=0}^{\min(\alpha,\beta)} {_m}\mathcal{D}^k_{n,(\alpha,\beta)}$ and hence that $C_{k,(\alpha,\beta)}(t) = \sum_{i=0}^{\min(\alpha,\beta)} {_m}C_{k,(\alpha,\beta)}(t)$ for all $k \geq 2$ and all shapes $(\alpha,\beta)$.  By construction, we also have ${_m}\mathcal{D}^k_{n,(\alpha,\beta)} = \mathcal{L}^{k,m}_{n,(\alpha,\beta)} - \mathcal{L}^{k,m+1}_{n,(\alpha,\beta)}$.  Using the bijection for the $\mathcal{L}^{k,m}_{n,(\alpha,\beta)}$ mentioned above, this final fact gives:

\begin{proposition}
\label{thm3: minimum height proposition}
For any $k \geq 2$, $n,\alpha,\beta \geq 0$ and $0 \leq m \leq \min(\alpha,\beta)$,
$${_m}C^k_{n,(\alpha,\beta)} = C^k_{n,(\alpha-m,\beta-m)} - C^k_{n,(\alpha-m-1,\beta-m-1)}.$$
\end{proposition}

The drawback with Proposition \ref{thm3: minimum height proposition} is that it relies upon the extremely lengthy formula of Theorem \ref{thm2: closed formula}.  This motivates the alternative characterization of ${_m}C^k_{n,(\alpha,\beta)}$ given below, which has the added benefit of relating all our results to enumerations of raised $k$-Dyck paths of shape $(\alpha,0)$.

\begin{theorem}
\label{thm3: minimum height generating function}
For any $k \geq 2$, $\alpha,\beta \geq 0$ and $0 \leq m \leq \min(\alpha,\beta)$,
$${_m}C_{k,(\alpha,\beta)}(t) = C_{k,(\alpha-m,0)}(t) \kern+1pt C_k(t)^{\beta-m}.$$
\end{theorem}
\begin{proof}
As shown in Figure \ref{fig3: minimum height decomposition}, every path $P \in {_m}\mathcal{D}^k_{n,(\alpha,\beta)}$ may be decomposed according to the rightmost point at its minimum height of $y=m$.  When $0 \leq m < \beta$, this decomposition gives the relationship ${_m}C_{k,(\alpha,\beta)}(t) = C_{k,(\alpha-m,0)}(t) \kern+1pt C_{k,(0,\beta-m-1)}(t)$.  When $m = \beta$, we have the relationship ${_m}C_{k,(\alpha,\beta)}(t) = C_{k,(\alpha-m,0)}(t)$.  Both cases simplify to the stated equation.
\end{proof}

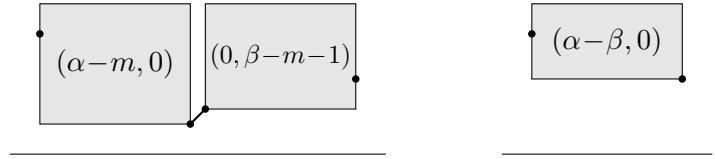
\begin{figure}[ht!]
\centering
\begin{tikzpicture}
	[scale=.4,auto=left,every node/.style={circle,fill=black,inner sep=0pt,outer sep=0pt}]
	\draw (0,0) to (12.5,0);
	\draw[fill=gray!20!](1,1) rectangle (6,5);
	\node[fill=none] (b1) at (3.5,3) {\scalebox{1}{$(\alpha \kern-2pt - \kern-2pt m,0)$}};
	\draw[thick] (6,1) to (6.5,1.5);
	\draw[fill=gray!20!](6.5,1.5) rectangle (11.5,5);
	\node[fill=none] (b2) at (9,3.25) {\scalebox{.9}{$(0,\beta \kern-2pt - \kern-2pt m \kern-2pt - \kern-2pt 1)$}};	
	\node[inner sep=1pt] (a) at (1,4) {};
	\node[inner sep=1pt] (b) at (6,1) {};
	\node[inner sep=1pt] (b) at (6.5,1.5) {};
	\node[inner sep=1pt] (b) at (11.5,2.5) {};
\end{tikzpicture}
\hspace{.5in}
\begin{tikzpicture}
	[scale=.4,auto=left,every node/.style={circle,fill=black,inner sep=0pt,outer sep=0pt}]
	\draw (0,0) to (7,0);
	\draw[fill=gray!20!](1,2.5) rectangle (6,5);
	\node[fill=none] (b1) at (3.5,3.75) {\scalebox{1}{$(\alpha \kern-2pt - \kern-2pt \beta,0)$}};
	\node[inner sep=1pt] (a) at (1,4) {};
	\node[inner sep=1pt] (b) at (6,2.5) {};
\end{tikzpicture}
\caption{The two possible decompositions of a path $P \in {_m}\mathcal{D}^k_{n,(\alpha,\beta)}$ with minimum height $m$, one for $m < \beta$ (left) and one for $m = \beta$ (right).  The $(a_i,b_i)$ denote the effective shape of each subpath.}
\label{fig3: minimum height decomposition}
\end{figure}

Avoiding the even more involved calculation suggested by Proposition \ref{thm3: minimum height proposition}, we use Theorem \ref{thm3: minimum height generating function} to develop an (admittedly still inelegant) closed formula for the ${_m}C^k_{n,(\alpha,\beta)}$:

\begin{theorem}
\label{thm3: minimum height closed formula}
For any $k \geq 2$, $n,\alpha,\beta \geq 0$, and $m \leq 0 \leq \min(\alpha,\beta)$,
$${_m}C^k_{n,(\alpha,\beta)} = \left( \sum_{i \geq 0} (-1)^i \frac{2\beta - m + 1}{k(n-i) + 2\beta-m+1} \binom{k(n-i)+2\beta-m+1}{n-i} \binom{\alpha-(k-1)i}{i} \right)$$
$$- \left( \sum_{i \geq 0} (-1)^i \frac{\beta-m}{k(n-i)+\beta-m} \binom{k(n-i)+\beta-m}{n-i} \binom{\alpha-m-1-(k-1)i}{i} \right).$$
\end{theorem}
\begin{proof}
Applying \eqref{eq2: closed formula proof 2} from the proof of Theorem \ref{thm2: closed formula} to our identity from Theorem \ref{thm3: minimum height generating function}, we see that ${_m}C^k_{n,(\alpha,\beta)} = C_k(t)^{\beta-m} \kern+1pt C_{k,(\alpha-m,0)}(t)$ may be rewritten as
\begin{equation}
\label{eq3: minimum height closed formula proof}
C_k(t)^{2\beta-m+1} \sum_{i \geq 0} (-1)^i \binom{\alpha-m-(k-1)i}{i} t^i - C_k(t)^{\beta-m} \sum_{i \geq 0} (-1)^i \binom{\alpha-m-1-(k-1)i}{i} t^i.
\end{equation}
Recalling the standard identity $C_k(t)^r = \sum_{i \geq 0} \frac{r}{ki+r} \binom{ki+r}{i}t^i$, both of the terms from \eqref{eq3: minimum height closed formula proof} are transformed into convolutions, from which the two summations of the theorem may be extracted.
\end{proof}

As was the case in Section \ref{sec: raised Dyck paths}, all of these formulas become much simpler when we restrict our attention to small $\alpha$ or to $k=2$.  When $\alpha - m \leq k-1$ we may apply Corollary \ref{thm2: closed formula corollary 2} to the $C_{k,(\alpha-m,0)}(t)$ term from Theorem \ref{thm3: minimum height generating function} and derive the following:

\begin{corollary}
\label{thm3: minimum height corollary 1}
For any $k \geq 2$, $n,\beta \geq 0$, and $m,\alpha \geq 0$ such that $0 \leq \alpha - m \leq k-1$,
$${_m}C^k_{n,(\alpha,\beta)} =
\begin{cases}
\frac{\beta-m+1}{kn+\beta-m+1} \binom{kn+\beta-m+1}{n} & \text{ if } m = \alpha,\\
\frac{\beta-m+1}{kn+\beta-m+1} \binom{kn+\beta-m+1}{n} - \frac{\beta-m}{kn+\beta-m} \binom{kn+\beta-m}{n} & \text{ if } m < \alpha \leq k-1+m. 
\end{cases}$$
\end{corollary}
\begin{proof}
By Corollary \ref{thm2: closed formula corollary 2} and Theorem \ref{thm3: minimum height generating function}, when $m = \alpha$ we have $C_{k,(\alpha-m,0)}(t) = C_k(t)$ and thus that ${_m}C_{k,(\alpha,\beta)}(t) = C_k(t)^{\beta-m+1}$.  Similarly, when $m < \alpha < k-1+m$ we have $C_{k,(\alpha-m,0)}(t) = C_k(t)-1$ and thus that ${_m}C_{k,(\alpha,\beta)}(t) = C_k(t)^{\beta-m+1} - C_k(t)^{\beta-m}$.  Our closed formulas then follow from the identity $[t^n]C_k(t)^r = \frac{r}{kn+r} \binom{kn+r}{n}$.
\end{proof}

As for the $k=2$ case, in the course of proving Theorem \ref{thm2: k=2 generating function} we already enumerated paths in $C^2_{n,(\alpha,\beta)}$ with minimal height $m$.  It may be verified that the formula below corresponds to $[t^n] t^{\alpha-m} C_2(t)^{\alpha+\beta+1-2m} = [t^n] C_{2,(\alpha-m)}(t) C_2(t)^{\beta-m}$, placing it in agreement with Theorem \ref{thm3: minimum height generating function}.

\begin{corollary}
\label{thm3: minimum height corollary 2}
For any $n,\alpha,\beta \geq 0$ and $0 \leq m \leq \min(\alpha,\beta)$,
$${_m}C^2_{n,(\alpha,\beta)} = \frac{\alpha+\beta+1-2m}{2n+\beta-\alpha+1} \binom{2n+\beta-\alpha+1}{n-\alpha+m}.$$
\end{corollary}

One unrelated consequence of Theorem \ref{thm3: minimum height generating function} is the following decomposition of $C_{k,(\alpha,\beta)}(t)$ into a sum that is indexed by minimal height:

\begin{equation}
\label{eq3: minimum height decomp of total generating function}
C_{k,(\alpha,\beta)}(t) = \sum_{i=0}^{\min(\alpha,\beta)} C_{k,(\alpha-m,0)}(t) \kern+1pt C_k(t)^{\beta-m}.
\end{equation}

Comparison of Proposition \ref{thm3: minimum height proposition} and Theorem \ref{thm3: minimum height generating function} also gives an unexpected equation whereby shape $(\alpha,\beta)$ paths may enumerated in terms of paths with shapes of the form $(\alpha',0)$ and $(0,\beta')$.

\begin{corollary}
\label{thm3: generating function in terms of (alpha,0)}
For all $k \geq 2$ and $\alpha,\beta \geq 0$,
$$C_{k,(\alpha,\beta)}(t) = \sum_{i=0}^{\min(\alpha,\beta)} C_{k,(\alpha-i,0)}(t) \kern+1pt C_k(t)^{\beta-i}.$$
\end{corollary}
\begin{proof}
Equating the right sides of Theorem \ref{thm3: minimum height generating function} and (a generating function-equivalent version of) Proposition \ref{thm3: minimum height proposition} when $m=0$ gives the relation below, which holds whenever $\alpha > 0$ and $\beta > 0$:
\begin{equation}
\label{eq3: generating function in terms of (alpha,0) proof}
C_{k,(\alpha,\beta)}(t) = C_{k,(\alpha,0)}(t) \kern+1pt C_k(t)^\beta + C_{k,(\alpha-1,\beta-1)}(t).
\end{equation}
Repeated application of this relation until $\alpha=0$ or $\beta=0$ yields the desired equation.
\end{proof}

\subsection{Raised k-Dyck paths, by returns}
\label{subsec: raised Dyck paths by returns}

Our next goal is to enumerate paths $P \in \mathcal{D}^k_{n,(\alpha,\beta)}$ with a specific number of ``returns to ground".  By a return to ground, we mean a $D$ step whose right endpoint lies on the line $y=0$.  When $\alpha=0$, the initial point $(0,0)$ of a path does not qualify as a return to ground.

Denote the set of all raised $k$-Dyck paths of length $kn+\beta-\alpha$ and shape $(\alpha,\beta)$ with precisely $\rho$ returns to ground by $\mathcal{D}^k_{n,(\alpha,\beta),\rho}$, and let $\vert \mathcal{D}^k_{n,(\alpha,\beta),\rho} \vert = C^k_{n,(\alpha,\beta),\rho}$.  As every path in $\mathcal{D}^k_{n,(\alpha,\beta)}$ contains precisely $n$ down steps, $C^k_{n,(\alpha,\beta),\rho} = 0$ if $\rho > n$.  When $\alpha > 0$, we may have $C^k_{n,(\alpha,\beta),\rho} = 0$ even if $\rho \leq n$.

In this section it is once again beneficial to preemptively fix a shape $(\alpha,\beta)$ and deal with the generating functions $C_{k,(\alpha,\beta)}(t) = [q^\alpha r^\beta]C_k(q,r,t)$.  Filtering by the number of returns, we then define $C_{k,(\alpha,\beta)}(t,z) = \sum_{n,\rho \geq 0} C^k_{n,(\alpha,\beta),\rho} t^n z^\rho$.

In the classic case of $\alpha=0$, we quickly recap the standard result.  Here, every path in $\mathcal{D}^k_{n,(0,\beta),\rho}$ may be decomposed according to its returns as in Figure \ref{fig3: return decomposition traditional}.  This decomposition gives

\begin{proposition}
\label{thm3: generating function by returns, alpha=0}
For any $k \geq 2$ and $\beta \geq 0$,

$$C_{k,(0,\beta)}(t,z) = \sum_{i \geq 0} z^i t^i C_k(t)^{\beta + i(k-1)} .$$

\end{proposition}
\begin{proof}
For paths $P \in \mathcal{D}^k_{n,(0,\beta)}$ with precisely $\rho$ returns, the decomposition of Figure \ref{fig3: return decomposition traditional} yields the generating function $C_{k,(0,\beta-1)}(t) \left( t \kern+1pt C_{k,(0,k-2)}(t) \right)^\rho = t^\rho C_k(t)^\beta \left( C_k(t)^{k-1} \right)^\rho$.
\end{proof}

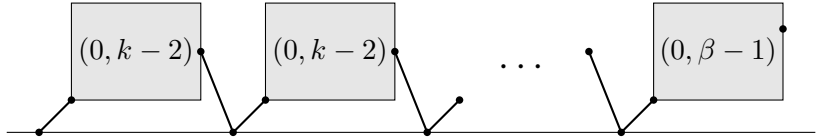
\begin{figure}[ht!]
\centering
\begin{tikzpicture}
	[scale=.43,auto=left,every node/.style={circle,fill=black,inner sep=0pt,outer sep=0pt}]
	\draw (-1,0) to (24,0);
	\draw[thick] (0,0) to (1,1);
	\draw[fill=gray!20!](1,1) rectangle (5,4);
	\node[fill=none] (b1) at (3,2.5) {\scalebox{1}{$(0,k-2)$}};
	\draw[thick] (5,2.5) to (6,0);
	\draw[thick] (6,0) to (7,1);
	\draw[fill=gray!20!](7,1) rectangle (11,4);
	\node[fill=none] (b2) at (9,2.5) {\scalebox{1}{$(0,k-2)$}};	
	\draw[thick] (11,2.5) to (12,0);
	\draw[thick] (12,0) to (13,1);
	\node[fill=none] (dots) at (15,2) {\scalebox{1.35}{$\cdots$}};
	\draw[thick] (17,2.5) to (18,0);
	\draw[thick] (18,0) to (19,1);
	\draw[fill=gray!20!](19,1) rectangle (23,4);
	\node[fill=none] (b3) at (21,2.5) {\scalebox{1}{$(0,\beta-1)$}};
	\node[inner sep=1pt] (1) at (0,0) {};
	\node[inner sep=1pt] (2) at (1,1) {};
	\node[inner sep=1pt] (3) at (5,2.5) {};
	\node[inner sep=1pt] (4) at (6,0) {};
	\node[inner sep=1pt] (5) at (7,1) {};
	\node[inner sep=1pt] (6) at (11,2.5) {};	
	\node[inner sep=1pt] (7) at (12,0) {};
	\node[inner sep=1pt] (8) at (13,1) {};
	\node[inner sep=1pt] (9) at (17,2.5) {};
	\node[inner sep=1pt] (10) at (18,0) {};
	\node[inner sep=1pt] (10) at (19,1) {};
	\node[inner sep=1pt] (11) at (23,3.2) {};
\end{tikzpicture}
\caption{The general form of a path $P \in \mathcal{D}^k_{n,(0,\beta)}$ with precisely $\rho$ returns to ground, along with the effective shape of each subpath.}
\label{fig3: return decomposition traditional}
\end{figure}

\begin{theorem}
\label{thm3: closed formula by returns, alpha=0}
For any $k \geq 2$ and $\beta,n,\rho \geq 0$,

$$C^k_{n,(0,\beta),\rho} = \frac{k\rho+\beta-\rho}{kn+\beta-\rho} \binom{kn+\beta-\rho}{n-\rho}.$$
\end{theorem}
\begin{proof}
By Proposition \ref{thm3: generating function by returns, alpha=0}, $C^k_{n,(0,\beta),\rho} = [t^n] t^\rho C_k(t)^{\beta + \rho(k-1)} = [t^{n-\rho}]C_k(t)^{\beta+\rho(k-1)}$. 
\end{proof}

The case of $\alpha > 0$ is similar yet slightly more complex, seeing as elements of $\mathcal{D}^k_{n,(\alpha,\beta),\rho}$ need not have a return to ground.  This necessitates two distinct decompositions for elements of $\mathcal{D}^k_{n,(\alpha,\beta),\rho}$, both of which are shown in Figure \ref{fig3: return decomposition raised}.  As with Proposition \ref{thm3: generating function by returns, alpha=0}, this decomposition prompts

\begin{proposition}
\label{thm3: generating function by returns, alpha>0}
For any $k \geq 2$ and $\beta \geq 0$ with $\alpha > 0$,

$$C_{k,(\alpha,\beta)}(t,z) = C_{k,(\alpha-1,\beta-1)}(t) + \sum_{i \geq 1} z^i t^i C_{k,(\alpha-1,k-2)}(t) \kern+1pt C_k(t)^{\beta + (i-1)(k-1)} .$$

\end{proposition}
\begin{proof}
The first term corresponds to the first decomposition in Figure \ref{fig3: return decomposition raised}.  The sum corresponds to the second decomposition in Figure \ref{fig3: return decomposition raised}, where paths $P \in \mathcal{D}^k_{n,(\alpha,\beta)}$ with $\rho$ returns have generating function $C_{k,(\alpha-1,k-2)}(t) \left(t \kern+1pt C_{k,(0,k-2)}(t) \right)^{\rho-1}  t \kern+1pt C_{k,(0,\beta-1)}(t) = t^\rho C_{k,(\alpha-1,k-2)}(t) \left( C_k(t)^{k-1} \right)^{\rho-1} C_k(t)^\beta$.
\end{proof}

\begin{figure}[ht!]
\centering
\begin{tikzpicture}
	[scale=.43,auto=left,every node/.style={circle,fill=black,inner sep=0pt,outer sep=0pt}]
	\draw (-1,0) to (5,0);
	\draw[fill=gray!20!](0,1) rectangle (4,4);
	\node[fill=none] (b1) at (2,2.5) {\scalebox{.8}{$(\alpha \kern-2pt - \kern-2pt 1,\beta \kern-2pt - \kern-2pt 1)$}};
	\node[inner sep=1pt] (1) at (0,2) {};
	\node[inner sep=1pt] (2) at (4,3.2) {};
\end{tikzpicture}
\hspace{.5in}
\begin{tikzpicture}
	[scale=.43,auto=left,every node/.style={circle,fill=black,inner sep=0pt,outer sep=0pt}]
	\draw (-1,0) to (24,0);
	\draw[fill=gray!20!](1,1) rectangle (5,4);
	\node[fill=none] (b1) at (3,2.5) {\scalebox{.8}{$(\alpha \kern-2pt - \kern-2pt 1,k \kern-2pt - \kern-2pt 2)$}};
	\draw[thick] (5,2.5) to (6,0);
	\draw[thick] (6,0) to (7,1);
	\draw[fill=gray!20!](7,1) rectangle (11,4);
	\node[fill=none] (b2) at (9,2.5) {\scalebox{1}{$(0,k-2)$}};	
	\draw[thick] (11,2.5) to (12,0);
	\draw[thick] (12,0) to (13,1);
	\node[fill=none] (dots) at (15,2) {\scalebox{1.35}{$\cdots$}};
	\draw[thick] (17,2.5) to (18,0);
	\draw[thick] (18,0) to (19,1);
	\draw[fill=gray!20!](19,1) rectangle (23,4);
	\node[fill=none] (b3) at (21,2.5) {\scalebox{1}{$(0,\beta-1)$}};
	\node[inner sep=1pt] (2) at (1,2) {};
	\node[inner sep=1pt] (3) at (5,2.5) {};
	\node[inner sep=1pt] (4) at (6,0) {};
	\node[inner sep=1pt] (5) at (7,1) {};
	\node[inner sep=1pt] (6) at (11,2.5) {};	
	\node[inner sep=1pt] (7) at (12,0) {};
	\node[inner sep=1pt] (8) at (13,1) {};
	\node[inner sep=1pt] (9) at (17,2.5) {};
	\node[inner sep=1pt] (10) at (18,0) {};
	\node[inner sep=1pt] (10) at (19,1) {};
	\node[inner sep=1pt] (11) at (23,3.2) {};
\end{tikzpicture}
\caption{The two possible decompositions for a path $P \in \mathcal{D}^k_{n,(\alpha,\beta)}$ with $\alpha > 0$, one for paths with no returns (left) and one for paths with precisely $\rho>0$ returns (right).}
\label{fig3: return decomposition raised}
\end{figure}
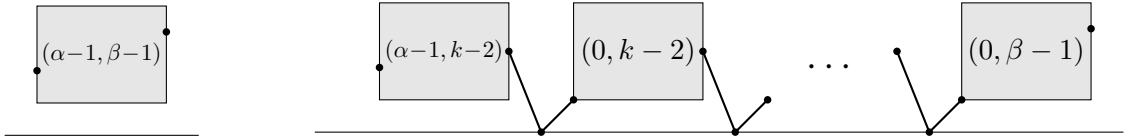

\begin{theorem}
\label{thm3: closed formula by returns, alpha>0}
For any $k \geq 2$ and $\beta,n,\rho \geq 0$ with $\alpha > 0$,

$$C^k_{n,(\alpha,\beta),\rho} =
\begin{cases}
C^k_{n,(\alpha-1,\beta-1)} & \text{ if } \rho=0,\\[+8pt]
\displaystyle{\sum_{i=0}^{n-\rho} C^k_{i,(\alpha-1,k-2)} R_{k,\beta+(\rho-1)(k-1)}(n-\rho-i)} & \text{ if } \rho >0.
\end{cases}$$
\end{theorem}
\begin{proof}
Using Proposition \ref{thm3: generating function by returns, alpha>0}, $C^k_{n,(\alpha,\beta),0} = [t^n]C_{k,(\alpha-1,\beta-1)}(t)$ when $\rho = 0$.  For $\rho > 0$ we have
\begin{equation}
\label{eq3: returns proof 1}
C^k_{n,(\alpha,\beta),\rho} = [t^n] t^\rho C_{k,(\alpha-1,k-2)}(t) \kern+1pt C_k(t)^{\beta + (\rho-1)(k-1)} = [t^{n-\rho}]C_{k,(\alpha-1,k-2)}(t) \kern+1pt C_k(t)^{\beta+(\rho-1)(k-1)}.
\end{equation}
\end{proof}

Given the complexity of the formula from Theorem \ref{thm2: closed formula}, substituting closed formulas into Theorem \ref{thm3: closed formula by returns, alpha>0} becomes very lengthy for arbitrary $(\alpha,\beta)$.  However, when $0 < \alpha \leq k$, we can apply Corollary \ref{thm2: closed formula corollary 1} (or Corollary \ref{thm2: closed formula corollary 2}) to arrive at the much simpler identity shown below.

\begin{corollary}
\label{thm3: closed formula by returns, corollary 1}
For any $k \geq 2$ and $\beta,n,\rho \geq 0$ with $0 < \alpha \leq k$,

$$C^k_{n,(\alpha,\beta),\rho} =
\begin{cases}
\frac{k\rho+\beta-\rho}{kn+\beta-\rho} \binom{kn+\beta-\rho}{n-\rho} & \text{ if } 0 < \alpha \leq k-1,\\[+12pt]
\frac{k\rho+\beta-\rho}{kn+\beta-\rho} \binom{kn+\beta-\rho}{n-\rho} - \frac{k \rho+\beta-\rho-(k-1)}{kn+\beta-\rho-(k-1)} \binom{kn+\beta-\rho-(k-1)}{n-\rho} & \text{ if } \alpha = k.
\end{cases}$$
\end{corollary}
\begin{proof}
Applying Corollary \ref{thm2: closed formula corollary 2} to the $C_{k,(\alpha-1,k-2)}(t)$ terms of Theorem \ref{thm3: closed formula by returns, alpha>0}, note that $\alpha - 1 \leq k-2$ implies $\alpha \leq k-1$, whereas $\alpha-1 > k-2$ along with $\alpha-1 \leq k -1$ together imply $\alpha = k$.
\end{proof}

As expected, the $k=2$ case is also comparatively succinct.  Not at all expected is that a specialization of Theorem \ref{thm3: closed formula by returns, alpha>0} to $k=2$ gives a simpler result when $\rho > 0$ than when $\rho = 0$.

\begin{corollary}
\label{thm3: closed formula by returns, corollary 2}
For any $\beta,n,\rho \geq 0$ with $\alpha > 0$,
$$C^k_{n,(\alpha,\beta),\rho} = 
\begin{cases}
\displaystyle{\sum_{i=0}^{\min(\alpha-1,\beta-1)}} \frac{\alpha+\beta-1-2i}{2n+\beta-\alpha+1} \binom{2n+\beta-\alpha+1}{n-\alpha+1+i} & \text{ if } \rho = 0, \\[+20pt]
\displaystyle{\frac{\alpha+\beta+\rho-1}{2n+\beta-\alpha-\rho+1} \binom{2n+\beta-\alpha-\rho+1}{n-\alpha-\rho+1}} & \text{ if } \rho > 0.
\end{cases}$$
\end{corollary}
\begin{proof}
The $\rho = 0$ case follows immediately from an application of Theorem \ref{thm2: k=2 closed formula} to Theorem \ref{thm3: closed formula by returns, alpha>0}.  For the $\rho > 0$ case, by Proposition \ref{thm2: k=2 horizontal reflection} we have $C^2_{n,(\alpha,\beta),\rho} = [t^n]t^\rho \kern+1pt C_{2,(\alpha-1,0)}(t) \kern+1pt C_2(t)^{\beta+\rho-1}$.  Using Proposition \ref{thm3: generating function by returns, alpha>0} then gives the following, to which we apply the definition of Raney numbers.
\begin{equation}
\label{eq3: closed formula by returns, corollary 2 proof 1}
C^2_{n,(\alpha,\beta),\rho} = [t^n] t^{\rho + \alpha - 1} C_{2,(0,\alpha-1)}(t) \kern+1pt C_2(t)^{\beta + \rho - 1} = [t^{n-\rho-\alpha+1}]C_2(t)^{\alpha+\beta+\rho-1}.
\end{equation}
\end{proof}

\section{Raised k-Dyck paths of bounded height}
\label{sec: Dyck paths of bounded height}

The results of Section \ref{sec: raised Dyck paths} may also be used to enumerate (raised) $k$-Dyck paths of bounded height.  This allows for a derivation of easily-computable generating functions that hold for all $k \geq 2$ and shapes $(\alpha,\beta)$, expanding upon the discussions of non-raised, height-bounded lattice paths in Baril and Prodinger \cite{BP}, Bousquet-M\'{e}lou \cite{Bousquet}, or Bacher \cite{Bacher}. 

So take any raised $k$-Dyck path $P \in \mathcal{D}^k_{n,(\alpha,\beta)}$.  If $P$ stays weakly below $y=M$, we say that $P$ is bounded from above by $M$.  We use $\mathcal{U}^{k,M}_{n,(\alpha,\beta)}$ to denote the collection of all $P \in \mathcal{D}^k_{n(\alpha,\beta)}$ that are bounded from above by $M$, and set $\vert \mathcal{U}^{k,M}_{n,(\alpha,\beta)} \vert = U^{k,M}_{n,(\alpha,\beta)}$.  Clearly, $U^{k,M}_{n,(\alpha,\beta)} = 0$ unless $\alpha,\beta \leq M$.

Fixing $0 \leq \alpha,\beta \leq M$, we define the generating function $U^M_{k,(\alpha,\beta)}(t) = \sum_{n \geq 0} U^{k,M}_{n,(\alpha,\beta)} t^n$.  The primary goal of this section is to relate the $U^M_{k,(\alpha,\beta)}(t)$ to the generating functions $C_{k,(\alpha',\beta')}(t)$ of Section \ref{sec: raised Dyck paths}, from which one may derive closed formulas for the $U^{k,M}_{n,(\alpha,\beta)}$ using Theorem \ref{thm2: closed formula}.

Before deriving a relationship for general $U^M_{k,(\alpha,\beta)}(t)$, we consider the special case of $\beta = M$:

\begin{lemma}
\label{thm4: bounded height paths lemma}
For any $k \geq 2$, $\alpha \geq 0$, and $M \geq 0$,

$$U^M_{k,(\alpha,M)}(t) = \frac{C_{k,(\alpha,M)}(t)}{1+C_{k,(M+1,M)}(t)}.$$
\end{lemma}
\begin{proof}
Every path $P \in \mathcal{D}^k_{n,(\alpha,M)}$ may be decomposed in one of the two ways shown in Figure \ref{fig4: bounded height decomposition 1}, depending upon whether or not the path rises above $y = M$.  This prompts the identity

\begin{equation}
\label{eq4: bounded heigh paths lemma}
C_{k,(\alpha,M)}(t) = U^M_{k,(\alpha,M)}(t) + U^M_{k,(\alpha,M)}(t) \kern+1pt C_{k,(M+1,M)}(t).
\end{equation}
 
\end{proof}

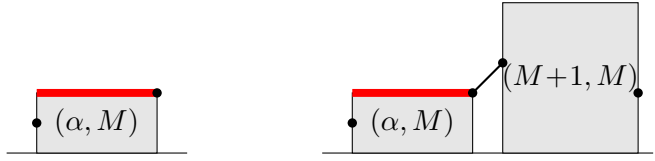
\begin{figure}[ht!]
\centering
\begin{tikzpicture}
	[scale=.4,auto=left,every node/.style={circle,fill=black,inner sep=0pt,outer sep=0pt}]
	\draw (-1,0) to (5,0);
	\draw[fill=gray!20!](0,0) rectangle (4,2);
	\draw[line width=3pt,color=red] (0,2) to (4,2);
	\node[fill=none] (b1) at (2,1) {\scalebox{1}{$(\alpha,M)$}};
	\node[inner sep=1.25pt] (a) at (0,1) {};
	\node[inner sep=1.25pt] (b) at (4,2) {};
\end{tikzpicture}
\hspace{.6in}
\begin{tikzpicture}
	[scale=.4,auto=left,every node/.style={circle,fill=black,inner sep=0pt,outer sep=0pt}]
	\draw (-1,0) to (10,0);
	\draw[fill=gray!20!](0,0) rectangle (4,2);
	\draw[line width=3pt,color=red] (0,2) to (4,2);
	\node[fill=none] (b1) at (2,1) {\scalebox{1}{$(\alpha,M)$}};
	\draw[thick] (4,2) to (5,3);
	\draw[fill=gray!20!](5,0) rectangle (9.5,5);	
	\node[fill=none] (b2) at (7.25,2.5) {\scalebox{1}{$(M \kern-2pt + \kern-2pt 1,M)$}};
	\node[inner sep=1.25pt] (a) at (0,1) {};
	\node[inner sep=1.25pt] (b) at (4,2) {};
	\node[inner sep=1.25pt] (b) at (5,3) {};
	\node[inner sep=1.25pt] (b) at (9.5,2) {};
\end{tikzpicture}
\caption{The two decompositions for a path $P \in \mathcal{D}^k_{n,(\alpha,M)}$, as used in the proof of Lemma \ref{thm4: bounded height paths lemma}.}
\label{fig4: bounded height decomposition 1}
\end{figure}

Lemma \ref{thm4: bounded height paths lemma} may still be applied to derive our general identity:

\begin{theorem}
\label{thm4: bounded height paths generating function}
For any $k \geq 2$, $M \geq 0$, and $0 \leq \alpha,\beta \leq M$,

$$U^M_{k,(\alpha,\beta)}(t) = C_{k,(\alpha,\beta)}(t) - \frac{C_{k,(\alpha,M)}(t) \kern+1pt C_{k,(M+1,\beta)}(t)}{1+C_{k,(M+1,M)}(t)}.$$
\end{theorem}
\begin{proof}
Via an equivalent decomposition of paths $P \in \mathcal{D}^k_{n,(\alpha,\beta)}$ to that in Figure \ref{fig4: bounded height decomposition 1}, we have

\begin{equation}
\label{eq4: bounded height paths generating function 1}
C_{k,(\alpha,\beta)}(t) = U^M_{k,(\alpha,\beta)}(t) + U^M_{k,(\alpha,M)}(t) \kern+1pt C_{k,(M+1,\beta)}(t).
\end{equation}

Rearranging \eqref{eq4: bounded height paths generating function 1} and applying Lemma \ref{thm4: bounded height paths lemma} then gives
\begin{equation}
\label{eq4: bounded height paths generating function 2}
U^M_{k,(\alpha,\beta)}(t) = C_{k,(\alpha,\beta)}(t) - U^M_{k,(\alpha,M)}(t) \kern+1pt C_{k,(M+1,\beta)}(t) = C_{k,(\alpha,\beta)}(t) - \frac{C_{k,(\alpha,M)}(t) \kern+1pt C_{k,(M+1,\beta)}(t)}{1+C_{k,(M+1,M)}(t)}.
\end{equation}
\end{proof}

Recall that the order of $C_{k,(\alpha,\beta)}(t)$ goes to $\infty$ and $\alpha \rightarrow \infty$.  This implies that the order of $C_{k,(\alpha,M)}(t) \kern+1pt C_{k,(M+1,\beta)}(t)$ goes to $\infty$ as $M \rightarrow \infty$, and thus that the order of $\frac{C_{k,(\alpha,M)}(t) \kern+1pt C_{k,(M+1,\beta)}(t)}{1+C_{k,(M+1,M)}(t)}$ goes to $\infty$ as $M \rightarrow \infty$.  This allows us to conclude that number of initial terms for which $[t^n]U^M_{k,(\alpha,\beta)}(t) = [t^n]C_{k,(\alpha,\beta)}(t)$ goes to $\infty$ and $M \rightarrow \infty$, as one would expect for $k$-Dyck paths with an arbitrarily high upper bound.

Also observe that, if $M < k-1$, then we have both $M+1 \leq k-1$ and $\alpha,\beta \leq M$.  This means that we can apply Corollary \ref{thm2: closed formula corollary 2} to the rightmost term from Theorem \ref{thm4: bounded height paths generating function} as below:

\begin{equation}
\label{eq4: bounded by M < k-1 case}
\frac{C_{k,(\alpha,M)}(t) \kern+1pt C_{k,(M+1,\beta)}(t)}{1+C_{k,(M+1,M)}(t)} = \frac{C_k(t)^{M+1} \left( C_k(t)^{\beta+1} - 1 \right)}{C_k(t)^{M+1}} = C_k(t)^{\beta+1} - 1.
\end{equation}

When $M < k-1$ and $\alpha \leq \beta$, this gives the expected result of $U^M_{k,(\alpha,\beta)}(t) = C_k(t)^{\beta+1} - \left( C_k(t)^{\beta+1} - 1 \right) = 1$, corresponding to the fact that only the ``trivial" path (i.e.- the unique path with zero $D$ steps) stays weakly below $y = M$ when $M < k-1$.  When $M < k-1$ and $\alpha > \beta$, we similarly get the expected result of $U^M_{k,(\alpha,\beta)}(t) = 0$, reflecting the fact that every path of such a shape $(\alpha,\beta)$ must have at least one $D$ step and thus can't stay weakly below $y = M$.

Explicit calculations involving the generating function $U^M_{k,(\alpha,\beta)}(t)$ become increasingly difficult when $M \geq k-1$, but Theorem \ref{thm4: bounded height paths generating function} may always be used with with Theorem \ref{thm2: closed formula} to calculate the cardinalities $U^{k,M}_{n,(\alpha,\beta)} = [t^n] U^M_{k,(\alpha,\beta)}(t)$.  See Appendix \ref{sec: appendix} for explicit calculations of the sequences generated by the $U^M_{k,(\alpha,\beta)}(t)$ for various $k \geq 2$ and small $M$ in the case of $(\alpha,\beta) = (0,0)$.
 
For one final application, note that Theorem \ref{thm4: bounded height paths generating function} may be used to enumerate the number of raised $k$-Dyck paths that actually obtain a fixed maximum height.  This follows immediately from the fact that raised $k$-Dyck paths of maximum height $M$ are precisely those paths that stay weakly below $y = M$ yet fail to stay weakly below $y = M-1$.

So let $\mathcal{H}^{k,M}_{n,(\alpha,\beta)}$ denote the set of all $P \in \mathcal{D}^k_{n,(\alpha,\beta)}$ that obtain a maximum height of $M$, and let $\vert \mathcal{H}^{k,M}_{n,(\alpha,\beta)} \vert = H^{k,M}_{n,(\alpha,\beta)}$.  In terms of the generating function $H^M_{k,(\alpha,\beta)}(t) = \sum_{n \geq 0} H^{k,M}_{n,(\alpha,\beta)} t^n$, Theorem \ref{thm4: bounded height paths generating function} immediately yields the following result.

\begin{corollary}
\label{thm4: maximum height paths generating function}
For any $k \geq 2$, $M \geq 0$ and $0 \leq \alpha,\beta \leq M$,
$$H^M_{k,(\alpha,\beta)}(t) = U^M_{k,(\alpha,\beta)}(t) - U^{M-1}_{k,(\alpha,\beta)}(t) = \frac{C_{k,(\alpha,M-1)}(t) \kern+1pt C_{k,(M,\beta)}(t)}{1 + C_{k,(M,M-1)}(t)} - \frac{C_{k,(\alpha,M)}(t) \kern+1pt C_{k,(M+1,\beta)}(t)}{1 + C_{k,(M+1,M)}(t)}.$$

\end{corollary}

As with the $U^M_{k,(\alpha,\beta)}(t)$, the $H^M_{k,(\alpha,\beta)}(t)$ become increasing exhausting to calculate when $M$ becomes large.  For $M < k-1$, it is still easy to verify that we get the expected results of $H^M_{k,(\alpha,\beta)}(t) = 1$ when $\alpha \leq \beta$ and $H^M_{k,(\alpha,\beta)}(t) = 0$ when $\alpha > \beta$.  See Appendix \ref{sec: appendix} for explicit calculations of the sequences generated by the $H^M_{k,(\alpha,\beta)}(t)$ for various $k \geq 2$ and $M \geq k$ in the case of $(\alpha,\beta) = (0,0)$.

\newpage

\appendix

\section{Explicit Calculations}
\label{sec: appendix}

Below are comparisons of the sequences generated by $C_{n,(\alpha,\beta)}(t)$ to preexisting sequences on OEIS, for $k=2,3,4$.  All sequences were calculated on Maple 19 via \eqref{eq2: closed formula proof 2} from the proof Theorem \ref{thm2: closed formula}.  All listed sequences are identical up to shifting or the complete absence of (one or more) initial terms.

Notice how Proposition \ref{thm2: k=2 horizontal reflection} ensures that the $k=2$ table is symmetric along the main diagonal, whereas the $k=3,4$ tables are not symmetric along the main diagonal.  For all tables, Corollary \ref{thm2: closed formula corollary 2} ensures that all sequences with $\alpha \leq k-1$ correspond to convolutions of the $k-$Catalan numbers.

\begin{table}[ht!]
\centering
\def\arraystretch{1.3}
\scalebox{.9}{
\begin{tabular}{|c|c|c|c|c|c|c|c|c|}
\hline
 & \textbf{$\boldsymbol{\beta=0}$} & \textbf{$\boldsymbol{\beta=1}$} & \textbf{$\boldsymbol{\beta=2}$} & \textbf{$\boldsymbol{\beta=3}$} & \textbf{$\boldsymbol{\beta=4}$} & \textbf{$\boldsymbol{\beta=5}$} & \textbf{$\boldsymbol{\beta=6}$} & \textbf{$\boldsymbol{\beta=7}$} \\ \hline
\textbf{$\boldsymbol{\alpha=0}$} & \seqnum{A000108} & \seqnum{A000108} & \seqnum{A000245} & \seqnum{A002057} & \seqnum{A000340} & \seqnum{A003517} & \seqnum{A000588} & \seqnum{A003518} \\ \hline
\textbf{$\boldsymbol{\alpha=1}$} & \seqnum{A000108} & \seqnum{A000108} & \seqnum{A000245} & \seqnum{A002057} & \seqnum{A000340} & \seqnum{A003517} & \seqnum{A000588} & \seqnum{A003518} \\ \hline
\textbf{$\boldsymbol{\alpha=2}$} & \seqnum{A000245} & \seqnum{A000245} & \seqnum{A026012} & \seqnum{A026016} & \seqnum{A026013} & \seqnum{A026017} & \seqnum{A026014} & \seqnum{A026018} \\ \hline
\textbf{$\boldsymbol{\alpha=3}$} & \seqnum{A002057} & \seqnum{A002057} & \seqnum{A026016} & \seqnum{A026029} & \seqnum{A026026} & \seqnum{A026030} & \seqnum{A026027} & \seqnum{A026031} \\ \hline
\textbf{$\boldsymbol{\alpha=4}$} & \seqnum{A000340} & \seqnum{A000340} & \seqnum{A026013} & \seqnum{A026026} & -- & -- & -- & -- \\ \hline
\textbf{$\boldsymbol{\alpha=5}$} & \seqnum{A003517} & \seqnum{A003517} & \seqnum{A026017} & \seqnum{A026030} & -- & -- & -- & -- \\ \hline
\textbf{$\boldsymbol{\alpha=6}$} & \seqnum{A000588} & \seqnum{A000588} & \seqnum{A026014} & \seqnum{A026027} & -- & -- & -- & -- \\ \hline
\textbf{$\boldsymbol{\alpha=7}$} & \seqnum{A003518} & \seqnum{A003518} & \seqnum{A026018} & \seqnum{A026031} & -- & -- & -- & -- \\ \hline
\end{tabular}}
\caption{A comparison of the sequences generated by $C_{2,(\alpha,\beta)}(t)$ to preexisting sequences on OEIS.}
\label{table: k=2}
\end{table}

\begin{table}[ht!]
\centering
\def\arraystretch{1.3}
\scalebox{.9}{
\begin{tabular}{|c|c|c|c|c|c|c|}
\hline
 & \textbf{$\boldsymbol{\beta=0}$} & \textbf{$\boldsymbol{\beta=1}$} & \textbf{$\boldsymbol{\beta=2}$} & \textbf{$\boldsymbol{\beta=3}$} & \textbf{$\boldsymbol{\beta=4}$} & \textbf{$\boldsymbol{\beta=5}$} \\ \hline
\textbf{$\boldsymbol{\alpha=0}$} & \seqnum{A001764} & \seqnum{A006013} & \seqnum{A001764} & \seqnum{A006629} & \seqnum{A102893} & \seqnum{A006630} \\ \hline
\textbf{$\boldsymbol{\alpha=1}$} & \seqnum{A001764} & \seqnum{A006013} & \seqnum{A001764} & \seqnum{A006629} & \seqnum{A102893} & \seqnum{A006630} \\ \hline
\textbf{$\boldsymbol{\alpha=2}$} & \seqnum{A001764} & \seqnum{A006013} & \seqnum{A001764} & \seqnum{A006629} & \seqnum{A102893} & \seqnum{A006630} \\ \hline
\textbf{$\boldsymbol{\alpha=3}$} & \seqnum{A334680} & -- & \seqnum{A334680} & -- & -- & -- \\ \hline
\textbf{$\boldsymbol{\alpha=4}$} & \seqnum{A336945} & \seqnum{A030983} & \seqnum{A336945} & -- & -- & -- \\ \hline
\textbf{$\boldsymbol{\alpha=5}$} & \seqnum{A334976} & \seqnum{A334977} & \seqnum{A334976} & -- & -- & -- \\ \hline
\end{tabular}}
\caption{A comparison of the sequences generated by $C_{3,(\alpha,\beta)}(t)$ to preexisting sequences on OEIS.}
\label{table: k=3}
\end{table}

\begin{table}[ht!]
\centering
\def\arraystretch{1.3}
\scalebox{.9}{
\begin{tabular}{|c|c|c|c|c|c|c|}
\hline
 & \textbf{$\boldsymbol{\beta=0}$} & \textbf{$\boldsymbol{\beta=1}$} & \textbf{$\boldsymbol{\beta=2}$} & \textbf{$\boldsymbol{\beta=3}$} & \textbf{$\boldsymbol{\beta=4}$} & \textbf{$\boldsymbol{\beta=5}$} \\ \hline
\textbf{$\boldsymbol{\alpha=0}$} & \seqnum{A002293} & \seqnum{A069271} & \seqnum{A006632} & \seqnum{A002293} & \seqnum{A196678} & \seqnum{A006633} \\ \hline
\textbf{$\boldsymbol{\alpha=1}$} & \seqnum{A002293} & \seqnum{A069271} & \seqnum{A006632} & \seqnum{A002293} & \seqnum{A196678} & \seqnum{A006633} \\ \hline
\textbf{$\boldsymbol{\alpha=2}$} & \seqnum{A002293} & \seqnum{A069271} & \seqnum{A006632} & \seqnum{A002293} & \seqnum{A196678} & \seqnum{A006633} \\ \hline
\textbf{$\boldsymbol{\alpha=3}$} & \seqnum{A002293} & \seqnum{A069271} & \seqnum{A006632} & \seqnum{A002293} & \seqnum{A196678} & \seqnum{A006633} \\ \hline
\textbf{$\boldsymbol{\alpha=4}$} & \seqnum{A334682} & -- & -- & \seqnum{A334682} & -- & --\\ \hline
\textbf{$\boldsymbol{\alpha=5}$} & -- & \seqnum{A334608} & -- & -- & -- & --\\ \hline
\end{tabular}}
\caption{A comparison of the sequences generated by $C_{4,(\alpha,\beta)}(t)$ to preexisting sequences on OEIS.}
\label{table: k=4}
\end{table}

\newpage

\noindent Below are comparisons of the sequences generated by $U^M_{k,(0,0)}(t)$ to preexisting sequences on OEIS, for $k=2,3,4$.  All sequences were calculated on Maple 19 via Theorem \ref{thm4: bounded height paths generating function}, and are identical to the listed sequences up to shifting or the absence of (one or more) initial terms.

\begin{table}[ht!]
\centering
\def\arraystretch{1.3}
\scalebox{.9}{
\begin{tabular}{|c|c|c|c|}
\hline
 & \textbf{$\boldsymbol{k=2}$} & \textbf{$\boldsymbol{k=3}$} & \textbf{$\boldsymbol{k=4}$} \\ \hline
\textbf{$\boldsymbol{M=0}$} & $1$ & $1$ & $1$ \\ \hline
\textbf{$\boldsymbol{M=1}$} & $\lbrace 1 \rbrace$ & $1$ & $1$ \\ \hline
\textbf{$\boldsymbol{M=2}$} & $\lbrace 2^n \rbrace$ & $\lbrace 1 \rbrace$ & $1$ \\ \hline
\textbf{$\boldsymbol{M=3}$} & \seqnum{A001519} & $\lbrace 2^n \rbrace$ & $\lbrace 1 \rbrace$ \\ \hline
\textbf{$\boldsymbol{M=4}$} & \seqnum{A124302} & $\lbrace 3^n \rbrace$ & $\lbrace 2^n \rbrace$ \\ \hline
\textbf{$\boldsymbol{M=5}$} & \seqnum{A080937} & \seqnum{A001835} & $\lbrace 3^n \rbrace$ \\ \hline
\textbf{$\boldsymbol{M=6}$} & \seqnum{A024175} & \seqnum{A081704} & $\lbrace 4^n \rbrace$ \\ \hline
\textbf{$\boldsymbol{M=7}$} & \seqnum{A080938} & \seqnum{A083881} & \seqnum{A004253} \\ \hline
\textbf{$\boldsymbol{M=8}$} & \seqnum{A033191} & -- & -- \\ \hline
\textbf{$\boldsymbol{M=9}$} & \seqnum{A211216} & -- & \seqnum{A261399} \\ \hline
\textbf{$\boldsymbol{M=10}$} & -- & -- & \seqnum{A143648} \\ \hline
\textbf{$\boldsymbol{M=11}$} & -- & -- & -- \\ \hline
\textbf{$\boldsymbol{M=12}$} & -- & -- & -- \\ \hline
\end{tabular}}
\caption{A comparison of the sequences generated by $U^M_{k,(0,0)}(t)$ to preexisting sequences on OEIS.  An entry of $1$ (without braces) corresponds to the sequence $1,0,0,0,\hdots$. }
\label{table: bounded from above}
\end{table}


\begin{thebibliography}{10}

\bibitem{AHS} A. Asinowski, B. Hackl and S. J. Selkirk, Down-step statistics in generalized Dyck paths, preprint, YEAR.  Available at \url{http://arxiv.org/abs/2007.15562}.

\bibitem{Bacher} A. Bacher, Generalized Dyck paths of bounded height, preprint, 2018.  Available at \url{http://arxiv.org/abs/1303.2724}.

\bibitem{BP} J.L. Basil and H. Prodinger, Enumeration of partial \L{}ukasiewicz paths, preprint, 2022.  Available at \url{http://arxiv.org/abs/2205.01383}.

\bibitem{Bousquet} M. Bousquet-M\'{e}lou, Discrete excursions, \textit{Sminaire Lotharingien de Combinatoire} \textbf{58} (2008), Article B57d.

\bibitem{HLM} S. Heubach, N.Y. Li and T. Mansour, Staircase tilings and $k$-Catalan structures, \textit{Discrete Math.} \textbf{308} (2008), no. 24, 5954--5964.

\bibitem{HP} P. Hilton and J. Pedersen, Catalan numbers, their generalizations, and their uses, \textit{Math. Intelligencer} \textbf{13} (1991), no. 2, 64--75.

\bibitem{Prodinger} H. Prodinger, On $k$-Dyck paths with a negative boundary, preprint, YEAR.  Available at \url{http://arxiv.org/abs/1912.06930}.

\bibitem{Selkirk} S.J. Selkirk, \textit{MSc-thesis: On a generalization of k-Dyck paths, Stellenbosch University}, 2019.

\bibitem{Stanley} R.P. Stanley, \textit{Catalan Numbers}, Cambridge University Press, 2015.

\bibitem{OEIS} OEIS Foundation Inc., The On-Line Encyclopedia of Integer Sequences, \url{https://oeis.org} (2019).

\end{thebibliography}
\end{document}